\documentclass[10pt,reqno,ptm,letterpaper]{amsart} 
\usepackage{amsmath,amssymb,rawfonts}
\newtheorem{thm}{Theorem}[section]
\begin{document} 
\title{UNIQUENESS OF SOLUTIONS TO BOLTZMANN EQUATIONS}
\author{Rafael Galeano Andrades$^{\lowercase{1}}$\,\ and \,\ Mario Almanza Caro$^{\lowercase{2}}$}
\address{Universidad de Cartagena, Facultad de Ciencias Exactas y Naturales}
\address{Departamento de Matemáticas.}
\email{$^1$rgaleanoa@unicartagena.edu.co}
\address{Universidad de Sucre, Facultad de Educaci\'on y Ciencias}
\address{Departamento de Matem\'aticas}
\email{$^2$mario.almanza@unisucre.edu.co}

\maketitle

\begin{abstract}
By means of Banach Fixed Point Theorem, the uniqueness of the Boltzmann equation generalized solutions in Sobolev spaces in $L^{1}\left(\mathbb{R}^{+}\times\mathbb{R}^{n}\times\mathbb{R}^{n}\right)$, can be proved,as well as the uniqueness of renormalized solutions of the Boltzmann equation.
\end{abstract}

\section{Introduction}
	
The problem is posed as follows: Let $f\left(t,x,v\right):\,\mathbb{R}^{+}\times\mathbb{R}^{n}\times\mathbb{R}^{n}\rightarrow\mathbb{R}^{+}$ measurable such that it satisfies: 
	
\begin{equation}
\left\{ \begin{array}{l}
\frac{\partial f}{\partial t}+v\cdot\nabla_{x}f=Q\left(f,f\right)\\
f\left(0,x,v\right)=f_{0}\left(x,v\right),
\end{array}\right.\label{eq:1}
\end{equation}

where 

\begin{equation}
Q\left(f,f\right)(v)=\int_{\mathbb{R}^{n}}\int_{\left|w\right|=1}B\left(w,\left|u-v\right|\right)\left[f\left(u'\right)f\left(v'\right)-f\left(u\right)f\left(v\right)\right]dwdu,
\end{equation}
	
as a result of the quadratic operator

\begin{equation}
Q\left(f,g\right)\left(v\right)=\frac{1}{2}\int_{\mathbb{R}^{n}}\int_{\left|w\right|=1}B\left(w,\left|u-v\right|\right)\left[f\left(v'\right)g\left(u'\right)+f\left(u'\right)f\left(v'\right)\right.\left.-f\left(u\right)g\left(v\right)-f\left(v\right)g\left(u\right)\right]dwdu
\end{equation}

in the case where $f=g$. It will bemoreover assumed that  $Q\left(f,g\right)\left(v\right)=0$
si $f\neq g$, $u'+v'=u+v$ and $\left|u'\right|^{2}+\left|v'\right|^{2}=\left|u\right|^{2}+\left|v\right|^{2}$. $B$ satisfies the following properties:

\begin{itemize}
	\item[] $B_{0}$) $B$ just depends $\left|v-u\right|$ y $\left|v-u,w\right|$.
	\item[] $B_{1}$) $B\in L_{loc}^{\infty}\left(\mathbb{R}^{n},S^{2}\right)$.
	\item[] $B_{2}$) $B\left(v,w\right)\leq\frac{b_{1}\left|\left(v,w\right)\right|}{\left|v\right|}\left(1+\left|v\right|^{\mu}\right)$.
	\item[] $B_{3}$) $\int_{S^{2}}B\left(v,w\right)dw\geq b\left|v\right|\left(1+\left|v\right|^{\mu}\right)^{-1}$.
\end{itemize}

	
A classical solution $f$ of (\ref{eq:1}) is a function $f\in C^{1}\left(\left[0,T\right]\times\mathbb{R}^{n}\times\mathbb{R}^{n}\right)$
that satisfies (\ref{eq:1}).
	
Let $f^{\#}\left(t,x,v\right)=f\left(t,x+vt,v\right)$, as a result

\begin{equation}
 \left\{ \begin{array}{lc}
             \frac{df}{dt}\left(t,x,v\right) =\frac{\partial f^{\#}}{\partial t}\left(t,x,v\right)+v\cdot\nabla_{x}f\left(t,x+vt,v\right)=Q^{\#}\left(f,f\right)(v), \\
             f^{\#}\left(0,x,v\right) =f\left(0,x,v\right)=f_{0}\left(x,v\right).
             \end{array}
   \right.
\end{equation}

If, it is supposed ${\displaystyle f^{\#}\in L^{1}\left([0,T]\times\mathbb{R}^{n}\times\mathbb{R}^{n}\right)}$
y ${\displaystyle Q^{\#}\left(f,f\right)(v)\in L^{1}\left([0,T]\times\mathbb{R}^{n}\times\mathbb{R}^{n}\right)}$,

then 

\begin{equation}\label{eq:2}
f^{\#}\left(t,x,v\right)=f_{0}\left(x,v\right)+\int_{0}^{t}Q^{\#}\left(t,f\right)(v)d\tau.
\end{equation}

The functions satisfying (\ref{eq:2}), are called generalized solutions of  (\ref{eq:1}). A solution ${\displaystyle f\in L^{1}\left(\mathbb{R}^{+}\times\mathbb{R}^{n}\times\mathbb{R}^{n}\right)}$ is called a renormalized solution of the Boltzmann equation if:
	
\[
\frac{Q\left(f,f\right)}{1+f}\in L^{1}\left(\mathbb{R}^{+}\times\mathbb{R}^{n}\times\mathbb{R}^{n}\right),
\]

and if for any continuous Lipchitz function $\beta:\,\mathbb{R}^{+}\rightarrow\mathbb{R}$
satisfying $\left|\beta'(t)\right|\leq\frac{c}{1+t}$ for all $t\geq0$, $c>0$, it follows:
	
\begin{equation}
\frac{\partial\beta}{\partial t}+v\cdot\nabla_{x}\beta(f)=\beta'(f)Q\left(f,f\right)\label{eq:3}
\end{equation}

in the sense of the distributions. In the particular case where $\beta(f)=f$, it is reduced to the problem studied by Diperna-Lions. In the particular case where $\beta(t)=\ln(1+t)$, $f$ is a Boltzmann's classical solution.
	
The first existence theorem on time-dependent Boltzmann equation solutions was released in 1932 when Carleman proved the existence of global solutions for the Cauchy problem in the especial homogeneous case (the solution does not depend on $x$). In 1963, Grad built the first local result close to the Maxwellian and in 1974, Ukai built global solutions close to the Maxwellian, extending the result of Grad. The Grad procedure was used to build solutions in $L^{\infty}$ for initial value and boundary problems (for example see Ukai (1976), Nishida (1976), Shizuta (1977), Ukay and Asano (1983), Asano (1985)). In the $L^{1}$ scheme, Diperna and Lions (1989) constructed $L^{1}$-global solutions with minimal hypothesis about the initial data. These solutions are called renormalized solutions. In 2009, a result given by Gressmann, Strain and Arkiv, prove the existence of Boltzmann equation classic solutions. The Boltzmann equation with the term force, adds the term  $F\cdot\nabla_{v}f$ to the left side of the equation. In this sense, in 2011, Galeano and others proved the existence in $L^{1}$ for small data and in 2007, Galeano proved the existence of solutions close to the vacuum.
	
\section{Desarrollo}

Let ${\displaystyle f_{1}^{\#}\in L^{1}\left(\left[0,T\right]\times\mathbb{R}^{n}\times\mathbb{R}^{n}\right)}$ y ${\displaystyle f_{2}^{\#}\in L^{1}\left(\left[0,T\right]\times\mathbb{R}^{n}\times\mathbb{R}^{n}\right)}$ 	be Boltzmann equation generalized solution and let  ${\displaystyle g=f_{1}^{\#}-f_{2}^{\#}\in L^{1}\left(\left[0,T\right]\times\mathbb{R}^{n}\times\mathbb{R}^{n}\right)}$ be   solutions to the generalized Boltzmann equation, that is to say:

\begin{align*}
f_{1}^{\#}\left(t,x,v\right) & =f_{0}\left(x,v\right)+\int_{0}^{t}Q^{\#}\left(f_{1},f_{1}\right)(v)d\tau,\\
f_{2}^{\#}\left(t,x,v\right) & =f_{0}\left(x,v\right)+\int_{0}^{t}Q^{\#}\left(f_{2},f_{2}\right)(v)d\tau.
\end{align*}

then:

\begin{eqnarray*}
f_{1}^{\#}\left(t,x,v\right)-f_{2}^{\#}\left(t,x,v\right)&=&f_{0}\left(x,v\right)+\int_{0}^{t}Q^{\#}\left(f_{1},f_{1}\right)(v)d\tau-f_{0}\left(x,v\right)\\
&&\hspace{1cm}-\int_{0}^{t}Q^{\#}\left(f_{2},f_{2}\right)(v)d\tau.
\end{eqnarray*}
    
That is to say: 

\begin{align*}
\mathcal{G} \left(t,x,v\right)&=\int_{0}^{t}Q^{\#}\left(f_{2}+g,f_{2}+g\right)(v)d\tau-\int_{0}^{t}Q^{\#}\left(f_{2},f_{2}\right)(v)d\tau,\\
\mathcal{G}(0,x,v)&=0.
\end{align*}

Let 

\[
F\left(g\right)\left(t,x,v\right):=\int_{0}^{t}Q^{\#}\left(f_{2}+g,f_{2}+g\right)(v)d\tau-\int_{0}^{t}Q^{\#}\left(f_{2},f_{2}\right)(v)d\tau
\]

be defined in  $L^{1}\left(\left[0,T\right]\times\mathbb{R}^{n}\times\mathbb{R}^{n}\right)$. It is noticed that $\mathcal{G}=0$ is a fixed point of $F$.It will be proved by means of Banach Fixed Theorem with some additional hypothesis, that zero is the only fixed point of $F$, in order to conclude that $\mathcal{G}\equiv0$ and therefore $f_{1}=f_{2}$.

\begin{thm}
If it is supposed $0\leq L<1$, such that 

\begin{eqnarray*}
\int_{\mathbb{R}^{n}}\int_{\left|w\right|=1}\left|f_{2}(u)\right|\left|B\left(u,\left|u-v\right|\right)\right|dwdu&\leq&\frac{L\left|v\right|}{b_{1}\left(1+\left|v\right|^{\mu}\right)},\\
\int_{\mathbb{R}^{n}}\int_{\left|w\right|=1}\left|g_{i}(u)\right|\left|B\left(w,\left|u-v\right|\right)\right|dwdu&\leq&\frac{L\left|v\right|}{b_{1}\left(1+\left|v\right|^{\mu}\right)},
\end{eqnarray*}

where $i=1,2$. Then $F$ is a continuous lipchitz.
\end{thm}
    
\begin{proof} \hspace{0.01cm}\\

$\displaystyle\left|F\left(g_{1}\right)\left(t,x,v\right)-F\left(g_{2}\right)\left(t,x,v\right)\right|$
\smallskip
\begin{eqnarray*}
&\leq&\int_{0}^{t}\left|Q^{\#}\left(f_{2}+g_{1},f_{2}+g_{1}\right)(v)-Q^{\#}\left(f_{2},f_{2}\right)\right.(v)\\
& &\hspace{0cm}\left.-Q^{\#}\left(f_{2}+g_{2},f_{2}+g_{2}\right)(v)+Q^{\#}\left(f_{2},f_{2}\right)(v)\right|d\tau\\
&=&\int_{0}^{t}\left|Q^{\#}\left(f_{2}+g_{1},f_{2}+g_{1}\right)(v)\right.\left.-Q^{\#}\left(f_{2}+g_{2},f_{2}+g_{2}\right)(v)\right|d\tau.
\end{eqnarray*}

Then 
\medskip

$\left|Q^{\#}\left(f_{2}+g_{1},f_{2}+g_{1}\right)(v)-Q^{\#}\left(f_{2}+g_{2},f_{2}+g_{2}\right)(v)\right|$

\begin{eqnarray*}
&=&\left|\int_{\mathbb{R}^{n}}\int_{\left|w\right|=1}B\left(w,\left|u-v\right|\right) \left[\left(f_{2}+g_{1}\right)\left(t,x+tv,u'\right)\left(f_{2}+g_{1}\right)\left(t,x+tv,v'\right)\right.\right.\\
&&\hspace{1.5cm}-\left(f_{2}+g_{1}\right)\left(t,x+tv,u\right)\left(f_{2}+g_{1}\right)\left(t,x+tv,v\right)\\
		& &\hspace{1.5cm}-\left.\left.\left(f_{2}+g_{2}\right)\left(t,x+tv,u'\right)\left(f_{2}+g_{2}\right)\left(t,x+tv,v'\right)\right.\right.\\
		&&\hspace{1.5cm}\left.+\left(f_{2}+g_{2}\right)\left(t,x+tv,u\right)\left(f_{2}+g_{2}\right)\left(t,x+tv,v\right)\right]dudw\Biggl|\\
&=&\left|\int_{\mathbb{R}^{n}}\int_{\left|w\right|=1}B\left(w,\left|u-v\right|\right)\left[\left(f_{2}\left(t,x+tv,u'\right)+g_{1}\left(t,x+tv,u'\right)\right)\left(f_{2}\left(t,x+tv,v'\right)+g_{1}\left(t,x+tv,v'\right)\right)\right.\right.\\
		&& -\left(f_{2}\left(t,x+tv,u\right)+g_{1}\left(t,x+tv,u\right)\right)\left(f_{2}\left(t,x+tv,v\right)+g_{1}\left(t,x+tv,v\right)\right)\\
		&&-\left(f_{2}\left(t,x+tv,u'\right)+g_{2}\left(t,x+tv,u'\right)\right)\left(f_{2}\left(t,x+tv,v'\right)+g_{2}\left(t,x+tv,v'\right)\right)\\
		&& +\left.\left.\left(f_{2}\left(t,x+tv,u\right)+g_{2}\left(t,x+tv,u\right)\right)\left(f_{2}\left(t,x+tv,v\right)+g_{2}\left(t,x+tv,v\right)\right)\right]dwdu\right.\Biggl|\\
&=& \left|\int_{\mathbb{R}^{n}}\int_{\left|w\right|=1}B\left(w,\left|u-v\right|\right)\left[\left(f_{2}^{\#}\left(t,x+t\left(v-u'\right),u'\right)+g_{1}^{\#}\left(t,x+t\left(v-u'\right),u'\right)\right)\right.\right.\\
&&\hspace{4cm}\left(f_{2}^{\#}\left(t,x+\left(v-v'\right),v'\right)+g_{1}^{\#}\left(t,x+t\left(v-v'\right),v'\right)\right)\\
		&& -\left.\left.\left(f_{2}^{\#}\left(t,x+\left(v-u\right),u\right)+g_{1}^{\#}\left(t,x+t\left(v-u\right),u\right)\right)\left(f_{2}^{\#}\left(t,x,v\right)+g_{1}^{\#}\left(t,x,v\right)\right)\right.\right.\\
		&& -\left(f_{2}^{\#}\left(t,x+\left(v-u'\right),u'\right)+g_{2}^{\#}\left(t,x+t\left(v-u'\right),u'\right)\right)\\
		&&\hspace{4cm}\left(f_{2}^{\#}\left(t,x+t\left(v-v'\right),v'\right)+g_{2}^{\#}\left(t,x+t\left(v-v'\right),v'\right)\right)\\
		&&
		+\left(f_{2}^{\#}\left(t,x+t\left(v-u\right),u\right)+g_{2}^{\#}\left(t,x+t\left(v-u\right),u\right)\right)\left(f_{2}^{\#}\left(t,x,v\right)+g_{2}^{\#}\left(t,x,v\right)\right)dwdu\Biggl|\\
& = & \left|\int_{\mathbb{R}^{n}}\int_{\left|w\right|=1}B\left(w,\left|u-v\right|\right)\left[f_{2}^{\#}\left(t,x+t\left(v-u'\right),u'\right)f_{2}^{\#}\left(t,x+t\left(v-v'\right),v'\right)\right.\right.\\
&&\hspace{1.5cm}+\left.\left.f_{2}^{\#}\left(t,x+t\left(v-u'\right),u'\right)g_{1}^{\#}\left(t,x+t\left(v-v'\right),v'\right)\right.\right.\\
&&\hspace{1.5cm}+g_{1}^{\#}\left(t,x+t\left(v-u'\right),u'\right)f_{2}^{\#}\left(t,x+t\left(v-v'\right),v'\right)\\
&&\hspace{1.5cm}+g_{1}^{\#}\left(t,x+t\left(v-u'\right),u'\right)g_{1}^{\#}\left(t,x+t\left(v-v'\right),v'\right)\\
&&\hspace{1.5cm}-f_{2}^{\#}\left(t,x+t\left(v-u\right),u\right)f_{2}^{\#}\left(t,x,v\right)-f_{2}^{\#}\left(t,x+t\left(v-u\right),u\right)g_{1}^{\#}\left(t,x,v\right)\\
&&\hspace{1.5cm}-g_{1}^{\#}\left(t,x+t\left(v-u\right),u\right)f_{2}^{\#}\left(t,x,v\right)-g_{1}^{\#}\left(t,x+t\left(v-u\right),u\right)g_{1}^{\#}\left(t,x,v\right)\\
&&\hspace{1.5cm}-f_{2}^{\#}\left(t,x+t\left(v-u'\right),u'\right)f_{2}^{\#}\left(t,x+t\left(v-v'\right),v'\right)\\
&&\hspace{1.5cm}-f_{2}^{\#}\left(t,x+t\left(v-u'\right),u'\right)g_{2}^{\#}\left(t,x+t\left(v-v'\right),v'\right)\\
&&\hspace{1.5cm}-g_{2}^{\#}\left(t,x+t\left(v-u'\right),u'\right)f_{2}^{\#}\left(t,x+t\left(v-v'\right),v'\right)\\
&&\hspace{1.5cm}-g_{2}^{\#}\left(t,x+t\left(v-u'\right),u'\right)g_{2}^{\#}\left(t,x+t\left(v-v'\right),v'\right)\\
&&\hspace{1.5cm}+f_{2}^{\#}\left(t,x+t\left(v-u\right),u\right)f_{2}^{\#}\left(t,x,v\right)+f_{2}^{\#}\left(t,x+t\left(v-u\right),u\right)g_{2}^{\#}\left(t,x,v\right)\\
&&\hspace{1.5cm}+g_{2}^{\#}\left(t,x+t\left(v-u\right),u\right)f_{2}^{\#}\left(t,x,v\right)+g_{2}^{\#}\left(t,x+t\left(v-u\right),u\right)g_{2}^{\#}\left(t,x,v\right)dwdu\Biggl|\\
&=&\left|\int_{\mathbb{R}^{n}}\int_{\left|w\right|=1}B\left(w,\left|u-v\right|\right) \left[f_{2}^{\#}\left(t,x+t\left(v-u'\right),u'\right)\left(g_{1}^{\#}\left(t,x+t\left(v-v'\right),v'\right)\right.\right.\right.\\
&&\hspace{5cm}\left.\left.-g_{2}^{\#}\left(t,x+t\left(v-v'\right),v'\right)\right)\right]dwdu\\
\end{eqnarray*}

\begin{eqnarray*}
&&+\int_{\mathbb{R}^{n}}\int_{\left|w\right|=1}B\left(w,\left|u-v\right|\right) \left[f_{2}^{\#}\left(t,x+t\left(v-v'\right),v'\right)\left(g_{1}^{\#}\left(t,x+t\left(v-u'\right),u'\right)\right.\right.\\
&&\hspace{5cm}\left.\left.-g_{2}^{\#}\left(t,x+t\left(v-u'\right),u'\right)\right)\right]dwdu\\
&&+\int_{\mathbb{R}^{n}}\int_{\left|w\right|=1}B\left(w,\left|u-v\right|\right) \left[f_{2}^{\#}\left(t,x+t\left(v-u\right),u\right)\left(g_{2}^{\#}\left(t,x,v\right)-g_{1}^{\#}\left(t,x,v\right)\right)\right]dwdu\\
&&+\int_{\mathbb{R}^{n}}\int_{\left|w\right|=1}B\left(w,\left|u-v\right|\right) \left[f_{2}^{\#}\left(t,x,v\right)\left(g_{2}^{\#}\left(t,x+t\left(v-u\right),u\right)-g_{1}^{\#}\left(t,x+t\left(v-u\right),u\right)\right)\right]dwdu\\
&&+\int_{\mathbb{R}^{n}}\int_{\left|w\right|=1}B\left(w,\left|u-v\right|\right) \left[g_{1}^{\#}\left(t,x+t\left(v-u'\right),u'\right)g_{1}^{\#}\left(t,x+t\left(v-v'\right),v'\right)\right.dwdu\\
&&\hspace{1.5cm}-g_{1}^{\#}\left(t,x+t\left(v-u\right),u\right)g_{1}^{\#}\left(t,x,v\right)-g_{2}^{\#}\left(t,x+t\left(v-u'\right),u'\right)g_{2}^{\#}\left(t,x+t\left(v-v'\right),v'\right)\\
&&\hspace{1.5cm}\left.+g_{2}^{\#}\left(t,x+t\left(v-u\right),u\right)g_{2}^{\#}\left(t,x,v\right)\right]dwdu\Biggl|\\
&=&\left|\int_{\mathbb{R}^{n}}\int_{\left|w\right|=1}B\left(w,\left|u-v\right|\right)f_{2}^{\#}\left(t,x+t\left(v-u'\right),u'\right)\left(g_{1}^{\#}\left(t,x+t\left(v-v'\right),v'\right)\right.\right.\\
&&\hspace{3cm}\left.\left.-g_{2}^{\#}\left(t,x+t\left(v-v'\right),v'\right)\right)\right.dwdu\\
&&+\left.\int_{\mathbb{R}^{n}}\int_{\left|w\right|=1}B\left(w,\left|u-v\right|\right)f_{2}^{\#}\left(t,x+t\left(v-v'\right),v'\right)\left(g_{1}^{\#}\left(t,x+t\left(v-u'\right),u'\right)\right.\right.\\
&&\hspace{3cm}\left.\left.-g_{2}^{\#}\left(t,x+t\left(v-u'\right),u'\right)\right)\right.dwdu\\
&&+\left.\int_{\mathbb{R}^{n}}\int_{\left|w\right|=1}B\left(w,\left|u-v\right|\right)f_{2}^{\#}\left(t,x+t\left(v-u\right),u\right)\left(g_{2}^{\#}\left(t,x,v\right)-g_{1}^{\#}\left(t,x,v\right)\right)\right.dwdu\\
&&+\left.\int_{\mathbb{R}^{n}}\int_{\left|w\right|=1}B\left(w,\left|u-v\right|\right)f_{2}^{\#}\left(t,x,v\right)\left(g_{2}^{\#}\left(t,x+t\left(v-u\right),u\right)-g_{1}^{\#}\left(t,x+t\left(v-u\right),u\right)\right)\right.dwdu\\
&&+\left.\int_{\mathbb{R}^{n}}\int_{\left|w\right|=1}B\left(w,\left|u-v\right|\right)\left[g_{1}^{\#}\left(t,x+t\left(v-u'\right),u'\right)g_{1}^{\#}\left(t,x+t\left(v-v'\right),v'\right)\right.\right.\\
&&\hspace{2cm}-g_{1}^{\#}\left(t,x+t\left(v-u'\right),u'\right)g_{2}^{\#}\left(t,x+t\left(v-v'\right),v'\right)\\
&&\hspace{2cm}+g_{1}^{\#}\left(t,x+t\left(v-u'\right),u'\right)g_{2}^{\#}\left(t,x+t\left(v-v'\right),v'\right)\\
&&\hspace{2cm}-g_{2}^{\#}\left(t,x+t\left(v-v'\right),v'\right)g_{2}^{\#}\left(t,x+t\left(v-u'\right),u'\right)\\
&&\hspace{2cm}+g_{2}^{\#}\left(t,x+t\left(v-u\right),u\right)g_{2}^{\#}\left(t,x,v\right)\\
&&\hspace{2cm}-g_{2}^{\#}\left(t,x+t\left(v-u\right),u\right)g_{1}^{\#}\left(t,x,v\right)+g_{2}^{\#}\left(t,x+t\left(v-u\right),u\right)g_{1}^{\#}\left(t,x,v\right)\\
&&\hspace{2cm}\left.\left.+g_{1}^{\#}\left(t,x+t\left(v-u\right),u\right)g_{1}^{\#}\left(t,x,v\right)\right]dwdu\right.\Biggl|\\
&=&\left|\int_{\mathbb{R}^{n}}\int_{\left|w\right|=1}B\left(w,\left|u-v\right|\right)f_{2}^{\#}\left(t,x+t\left(v-u'\right),u'\right)\left(g_{1}^{\#}\left(t,x+t\left(v-v'\right),v'\right)\right.\right.\\
&&\hspace{3cm}\left.-g_{2}^{\#}\left(t,x+t\left(v-v'\right),v'\right)\right)dwdu\\
&&+\int_{\mathbb{R}^{n}}\int_{\left|w\right|=1}B\left(w,\left|u-v\right|\right)f_{2}^{\#}\left(t,x+t\left(v-v'\right),v'\right)\left(g_{1}^{\#}\left(t,x+t\left(v-u'\right),u'\right)\right.\\
&&\hspace{3cm}\left.-g_{2}^{\#}\left(t,x+t\left(v-u'\right),u'\right)\right)dwdu\\
\end{eqnarray*}

\begin{eqnarray*}
&&+\int_{\mathbb{R}^{n}}\int_{\left|w\right|=1}B\left(w,\left|u-v\right|\right)f_{2}^{\#}\left(t,x+t\left(v-u\right),u\right)\left(g_{2}^{\#}\left(t,x,v\right)-g_{1}^{\#}\left(t,x,v\right)\right)dwdu\\
&&+\int_{\mathbb{R}^{n}}\int_{\left|w\right|=1}B\left(w,\left|u-v\right|\right)f_{2}^{\#}\left(t,x,v\right)\left(g_{2}^{\#}\left(t,x+t\left(v-u\right),u\right)\right.\\
&&\hspace{3cm}\left.-g_{1}^{\#}\left(t,x+t\left(v-u\right),u\right)\right)dwdu\\
&&+\int_{\mathbb{R}^{n}}\int_{\left|w\right|=1}B\left(w,\left|u-v\right|\right)g_{1}^{\#}\left(t,x+t\left(v-u'\right),u'\right)\left(g_{1}^{\#}\left(t,x+t\left(v-v'\right),v'\right)\right.\\
&&\hspace{3cm}\left.-g_{2}^{\#}\left(t,x+t\left(v-v'\right),v'\right)\right)dwdu\\
&&+\int_{\mathbb{R}^{n}}\int_{\left|w\right|=1}B\left(w,\left|u-v\right|\right)g_{2}^{\#}\left(t,x+t\left(v-v'\right),v'\right)\left(g_{1}^{\#}\left(t,x+t\left(v-u'\right),u'\right)\right.\\
&&\hspace{3cm}\left.-g_{2}^{\#}\left(t,x+t\left(v-u'\right),u'\right)\right)dwdu\\
&&+\int_{\mathbb{R}^{n}}\int_{\left|w\right|=1}B\left(w,\left|u-v\right|\right)g_{2}^{\#}\left(t,x+t\left(v-u\right),u\right)\left(g_{2}^{\#}\left(t,x,v\right)-g_{1}^{\#}\left(t,x,v\right)\right)dwdu\\
&&+\int_{\mathbb{R}^{n}}\int_{\left|w\right|=1}B\left(w,\left|u-v\right|\right)g_{1}^{\#}\left(t,x,v\right)\left(g_{2}^{\#}\left(t,x+t\left(v-u\right),u\right)-g_{1}^{\#}\left(t,x+t\left(v-u\right),u\right)\right)dwdu\\
\end{eqnarray*}

So 
\medskip

$\left\Vert F\left(g_{1}\right)-F\left(g_{2}\right)\right\Vert _{L^{1}\left([0,T]\times\mathbb{R}^{n}\times\mathbb{R}^{n}\right)}$

\begin{eqnarray*}
	&\leq&\int_{\mathbb{R}^{n}}\int_{\left|w\right|=1}\left|B\left(w,\left|u-v\right|\right)\right|f_{2}^{\#}\left(t,x+t\left(v-v'\right),v'\right)dwdu\left\Vert g_{1}^{\#}-g_{2}^{\#}\right\Vert _{L^{1}\left([0,T]\times\mathbb{R}^{n}\times\mathbb{R}^{n}\right)}\\
	&&+\int_{\mathbb{R}^{n}}\int_{\left|w\right|=1}\left|B\left(w,\left|u-v\right|\right)\right|f_{2}^{\#}\left(t,x+t\left(v-v'\right),v'\right)dwdu\left\Vert g_{1}^{\#}-g_{2}^{\#}\right\Vert _{L^{1}\left([0,T]\times\mathbb{R}^{n}\times\mathbb{R}^{n}\right)}\\
	&&+\int_{\mathbb{R}^{n}}\int_{\left|w\right|=1}\left|B\left(w,\left|u-v\right|\right)\right|f_{2}^{\#}\left(t,x,t\left(v-u\right),u\right)dwdv\left\Vert g_{1}^{\#}-g_{2}^{\#}\right\Vert _{L^{1}\left([0,T]\times\mathbb{R}^{n}\times\mathbb{R}^{n}\right)}\\
	&&+\int_{\mathbb{R}^{n}}\int_{\left|w\right|=1}\left|B\left(w,\left|u-v\right|\right)\right|f_{2}^{\#}\left(t,x,v\right)dwdu\left\Vert g_{1}^{\#}-g_{2}^{\#}\right\Vert _{L^{1}\left([0,T]\times\mathbb{R}^{n}\times\mathbb{R}^{n}\right)}\\
	&&+\int_{\mathbb{R}^{n}}\int_{\left|w\right|=1}\left|B\left(w,\left|u-v\right|\right)\right|g_{1}^{\#}\left(t,x+t\left(v-u'\right),u'\right)dwdu\left\Vert g_{1}^{\#}-g_{2}^{\#}\right\Vert _{L^{1}\left([0,T]\times\mathbb{R}^{n}\times\mathbb{R}^{n}\right)}\\
	&&+\int_{\mathbb{R}^{n}}\int_{\left|w\right|=1}\left|B\left(w,\left|u-v\right|\right)\right|g_{2}^{\#}\left(t,x+t\left(v-v'\right),v'\right)dwdu\left\Vert g_{1}^{\#}-g_{2}^{\#}\right\Vert _{L^{1}\left([0,T]\times\mathbb{R}^{n}\times\mathbb{R}^{n}\right)}\\
	&&+\int_{\mathbb{R}^{n}}\int_{\left|w\right|=1}\left|B\left(w,\left|u-v\right|\right)\right|g_{2}^{\#}\left(t,x+t\left(v-u\right),u\right)dwdu\left\Vert g_{1}^{\#}-g_{2}^{\#}\right\Vert _{L^{1}\left([0,T]\times\mathbb{R}^{n}\times\mathbb{R}^{n}\right)}\\
	&&+\int_{\mathbb{R}^{n}}\int_{\left|w\right|=1}\left|B\left(w,\left|u-v\right|\right)\right|g_{1}^{\#}\left(t,x,v\right)dwdu\left\Vert g_{1}^{\#}-g_{2}^{\#}\right\Vert _{L^{1}\left([0,T]\times\mathbb{R}^{n}\times\mathbb{R}^{n}\right)}.\\
\end{eqnarray*}

\begin{eqnarray*}
	&\leq&\frac{b_{1}}{\left|v\right|}\left(1+\left|v\right|^{\mu}\right)\int_{\mathbb{R}^{n}}\int_{\left|w\right|=1}\left|\left(v-u,w\right)\right|f_{2}^{\#}\left(t,x+t\left(v-v'\right),v'\right)dwdu\left\Vert g_{1}^{\#}-g_{2}^{\#}\right\Vert _{L^{1}\left([0,T]\times\mathbb{R}^{n}\times\mathbb{R}^{n}\right)}\\
	&&+\frac{b_{1}}{\left|v\right|}\left(1+\left|v\right|^{\mu}\right)\int_{\mathbb{R}^{n}}\int_{\left|w\right|=1}\left|\left(v-u,w\right)\right|f_{2}^{\#}\left(t,x+t\left(v-v'\right),v'\right)dwdu\left\Vert g_{1}^{\#}-g_{2}^{\#}\right\Vert _{L^{1}\left([0,T]\times\mathbb{R}^{n}\times\mathbb{R}^{n}\right)}\\
	&&+\frac{b_{1}}{\left|v\right|}\left(1+\left|v\right|^{\mu}\right)\int_{\mathbb{R}^{n}}\int_{\left|w\right|=1}\left|\left(v-u,w\right)\right|f_{2}^{\#}\left(t,x+t\left(v-u\right),u\right)dwdu\left\Vert g_{1}^{\#}-g_{2}^{\#}\right\Vert _{L^{1}\left([0,T]\times\mathbb{R}^{n}\times\mathbb{R}^{n}\right)}\\
	&&+\frac{b_{1}}{\left|v\right|}\left(1+\left|v\right|^{\mu}\right)\int_{\mathbb{R}^{n}}\int_{\left|w\right|=1}\left|\left(v-u,w\right)\right|f_{2}^{\#}\left(t,x,v\right)dwdu\left\Vert g_{1}^{\#}-g_{2}^{\#}\right\Vert _{L^{1}\left([0,T]\times\mathbb{R}^{n}\times\mathbb{R}^{n}\right)}\\
	&&+\frac{b_{1}}{\left|v\right|}\left(1+\left|v\right|^{\mu}\right)\int_{\mathbb{R}^{n}}\int_{\left|w\right|=1}\left|\left(v-u,w\right)\right|g_{1}^{\#}\left(t,x+t\left(v-u'\right),u'\right)dwdu\left\Vert g_{1}^{\#}-g_{2}^{\#}\right\Vert _{L^{1}\left([0,T]\times\mathbb{R}^{n}\times\mathbb{R}^{n}\right)}\\
	&&+\frac{b_{1}}{\left|v\right|}\left(1+\left|v\right|^{\mu}\right)\int_{\mathbb{R}^{n}}\int_{\left|w\right|=1}\left|\left(v-u,w\right)\right|g_{2}^{\#}\left(t,x+t\left(v-v'\right),v'\right)dwdu\left\Vert g_{1}^{\#}-g_{2}^{\#}\right\Vert _{L^{1}\left([0,T]\times\mathbb{R}^{n}\times\mathbb{R}^{n}\right)}\\
	&&+\frac{b_{1}}{\left|v\right|}\left(1+\left|v\right|^{\mu}\right)\int_{\mathbb{R}^{n}}\int_{\left|w\right|=1}\left|\left(v-u,w\right)\right|g_{2}^{\#}\left(t,x+t\left(v-u\right),u\right)dwdu\left\Vert g_{1}^{\#}-g_{2}^{\#}\right\Vert _{L^{1}\left([0,T]\times\mathbb{R}^{n}\times\mathbb{R}^{n}\right)}\\
	&&+\frac{b_{1}}{\left|v\right|}\left(1+\left|v\right|^{\mu}\right)\int_{\mathbb{R}^{n}}\int_{\left|w\right|=1}\left|\left(v-u,w\right)\right|g_{1}^{\#}\left(t,x,v\right)dwdu\left\Vert g_{1}^{\#}-g_{2}^{\#}\right\Vert _{L^{1}\left([0,T]\times\mathbb{R}^{n}\times\mathbb{R}^{n}\right)}.\\
\end{eqnarray*}

Then, 
\[\left\Vert F\left(g_{1}\right)-F\left(g_{2}\right)\right\Vert _{L^{1}\left([0,T]\times\mathbb{R}^{n}\times\mathbb{R}^{n}\right)}\leq L\left\Vert g_{1}^{\#}-g_{2}^{\#}\right\Vert _{L^{1}\left([0,T]\times\mathbb{R}^{n}\times\mathbb{R}^{n}\right)}
\]

Therefore, by Banach's Fixed Point Theorem, operator  $F$ has a single Banach fixed point, that is, $\mathcal{G}\equiv0$, then $f_{1}=f_{2}$.
\end{proof}

\section{Renormalized solutions.}
	
\begin{thm}If it is supposed $\beta_{1}\neq\beta_{2}$, $\beta_{1}\geq\beta_{2}+\sqrt{c}$
y ${\displaystyle \int_{0}^{T}Q^{\#}\left(g,g\right)dt<1}$, then there is only one solution to the Boltzmann renormalized equation.
\end{thm}
	
\begin{proof} Let $f_{1}$ and $f_{2}$ be two renormalized solutions of Boltzmann's equation, and
	
	$$
	\frac{\partial\beta\left(f_{1}\right)}{\partial t}+v\cdot\nabla_{x}\beta\left(f_{1}\right)=\beta'\left(f_{1}\right)Q\left(f_{1},f_{1}\right),
	$$
	
	$$
	\frac{\partial\beta\left(f_{2}\right)}{\partial t}+v\cdot\nabla_{x}\beta\left(f_{2}\right)=\beta'\left(f_{2}\right)Q\left(f_{2},f_{2}\right),
	$$
	
	\[
	\frac{\partial}{\partial t}\left[\beta\left(f_{1}\right)-\beta\left(f_{2}\right)\right]+v\cdot\nabla_{x}\left[\beta\left(f_{1}\right)-\beta\left(f_{2}\right)\right]=\beta'\left(f_{1}\right)Q\left(f_{1},f_{1}\right)-\beta'\left(f_{2}\right)Q\left(f_{2},f_{2}\right).
	\]
	
Let $g=f_{1}-f_{2}$, so 

\begin{eqnarray*}
\frac{\partial}{\partial t}\left[\beta\left(f_{2}+g\right)-\beta\left(f_{2}\right)\right]+v\cdot\nabla_{x}\left[\beta\left(f_{2}+g\right)-\beta\left(f_{2}\right)\right]&=&\beta'\left(f_{2}+g\right)Q\left(f_{2}+g,f_{2}+g\right)\\
& &\hspace{0.5cm}-\beta'\left(f_{2}\right)Q\left(f_{2},f_{2}\right).
\end{eqnarray*}
	
Let

\[
\left[\beta^{\#}\left(f_{2}+g\right)-\beta^{\#}\left(f_{2}\right)\right]\left(t,x,v\right)=\beta\left(f_{2}+g\right)\left(t,x+vt,v\right)-\beta\left(f_{2}\right)\left(t,x+vt,v\right),
\]
so, 
\[
\frac{d}{dt}\left[\beta^{\#}\left(f_{2}+g\right)-\beta^{\#}\left(f_{2}\right)\right]=\beta'^{\#}\left(f_{2}+g\right)Q^{\#}\left(f_{2}+g,f_{2}+g\right)-\beta'^{\#}\left(f_{2}\right)Q^{\#}\left(f_{2},f_{2}\right).
\]

Then
\begin{eqnarray*}&&\beta^{\#}\left(f_{2}+g\right)\left(t,x,v\right)-\beta^{\#}\left(f_{2}\right)\left(t,x,v\right)\\
&& \hspace{1.7cm}=\beta^{\#}\left(f_{2}+g\right)\left(0,x,v\right)-\beta^{\#}\left(f_{2}\right)\left(0,x,v\right)\\
& &\hspace{1.7cm}+\int_{0}^{T}\left[\beta'^{\#}\left(f_{2}+g\right)Q^{\#}\left(f_{2}+g,f_{2}+g\right)-\beta'^{\#}\left(f_{2}\right)Q^{\#}\left(f_{2},f_{2}\right)\right]dt.
\end{eqnarray*}
	
If it is defined
\begin{eqnarray*}&&F\left[\left(\beta^{\#}\left(f_{2}+g\right)-\beta^{\#}\left(f_{2}\right)\right)\left(t,x+vt,v\right)\right]\\
&& \hspace{1.5cm}=\beta^{\#}\left(f_{2}+g\right)\left(0,x,v\right)-\beta^{\#}\left(f_{2}\right)\left(0,x,v\right)\\
& & \hspace{1.5cm}+\int_{0}^{T}\left[\beta'^{\#}\left(f_{2}+g\right)Q^{\#}\left(f_{2}+g,f_{2}+g\right)-\beta'^{\#}\left(f_{2}\right)Q^{\#}\left(f_{2},f_{2}\right)\right]dt,
\end{eqnarray*}

we notice that if $g=0$, then it is a critical point of $F$ . Let's verify that it is the only one, to conclude that  $f_{1}=f_{2}$. For this we will use in Banach's Fixed Point Theorem, in effect:
	
$F\left[\beta_{1}^{\#}\left(f_{2}+g\right)-\beta_{1}^{\#}\left(f_{2}\right)-\beta_{2}^{\#}\left(f_{2}+g\right)+\beta_{2}^{\#}\left(f_{2}\right)\right]$	
	
\begin{eqnarray*}
&=&F\left[\beta_{1}^{\#}\left(f_{2}+g\right)-\beta_{2}^{\#}\left(f_{2}+g\right)+\beta_{2}^{\#}\left(f_{2}\right)-\beta_{1}^{\#}\left(f_{2}\right)\right]\\
&=&\int_{0}^{T}\left[\beta_{1}'^{\#}\left(f_{2}+g\right)Q^{\#}\left(f_{2}+g,f_{2}+g\right)-\beta_{1}'^{\#}\left(f_{2}\right)Q^{\#}\left(f_{2},f_{2}\right)\right]dt\\
& &-\int_{0}^{T}\left[\beta_{2}'^{\#}\left(f_{2}+g\right)Q^{\#}\left(f_{2}+g,f_{2}+g\right)-\beta_{2}'^{\#}\left(f_{2}\right)Q^{\#}\left(f_{2},f_{2}\right)\right]dt\\
&=&\int_{0}^{T}\left[\beta_{1}'^{\#}\left(f_{2}+g\right)Q^{\#}\left(f_{2}+g,f_{2}+g\right)-\beta_{2}'^{\#}\left(f_{2}+g\right)Q^{\#}\left(f_{2}+g,f_{2}+g\right)\right]dt\\
  & &+\int_{0}^{T}\left[\beta_{2}'^{\#}\left(f_{2}\right)Q^{\#}\left(f_{2},f_{2}\right)-\beta_{1}'^{\#}\left(f_{2}\right)Q^{\#}\left(f_{2},f_{2}\right)\right]dt\\
&=&\int_{0}^{T}\left[\beta_{1}'^{\#}\left(f_{2}+g\right)-\beta_{2}'^{\#}\left(f_{2}+g\right)\right]Q^{\#}\left(f_{2}+g,f_{2}+g\right)dt\\
&&-\int_{0}^{T}\left[\beta_{2}'^{\#}\left(f_{2}\right)-\beta_{1}'^{\#}\left(f_{2}\right)\right]Q^{\#}\left(f_{2},f_{2}\right)dt\\
&=& \int_{0}^{T}\left[\beta_{1}{}^{\#}\left(f_{2}+g\right)-\beta_{2}{}^{\#}\left(f_{2}+g\right)\right]'Q^{\#}\left(f_{2}+g,f_{2}+g\right)dt\\
&&+\int_{0}^{T}\left[\beta_{2}\#\left(f_{2}\right)-\beta_{1}\#\left(f_{2}\right)\right]'Q^{\#}\left(f_{2},f_{2}\right)dt\\
&\leq&\int_{0}^{T}\frac{c}{1+\beta_{1}^{\#}\left(f_{2}+g\right)-\beta_{2}^{\#}\left(f_{2}+g\right)}Q^{\#}\left(f_{2}+g,f_{2}+g\right)dt\\
&&-\int_{0}^{T}\frac{c}{1+\beta_{2}^{\#}\left(f_{2}\right)-\beta_{1}^{\#}\left(f_{2}\right)}Q^{\#}\left(f_{2}+g,f_{2}+g\right)dt\\
\end{eqnarray*}

\begin{eqnarray*}
&=&\int_{0}^{T}\frac{c}{1+\left(\beta_{1}^{\#}-\beta_{2}^{\#}\right)}\left[Q^{\#}\left(f_{2}+g,f_{2}+g\right)-Q^{\#}\left(f_{2},f_{2}\right)\right]dt\\
&\leq&\int_{0}^{T}\frac{c}{\left(\beta_{1}^{\#}-\beta_{2}^{\#}\right)}\left[Q^{\#}\left(f_{2}+g,f_{2}+g\right)-Q^{\#}\left(f_{2},f_{2}\right)\right]dt\\
&\leq&\int_{0}^{T}\left(\beta_{1}^{\#}-\beta_{2}^{\#}\right)\left[Q^{\#}\left(f_{2}+g,f_{2}+g\right)-Q^{\#}\left(f_{2},f_{2}\right)\right]dt\,\,\,\,\textrm{(si}\,\,\,\beta_{1}\geq\beta_{2}+\sqrt{c}\textrm{).}
\end{eqnarray*}
	
So that 

$$\left\Vert F\left[\beta_{1}^{\#}\left(f_{2}+g\right)-\beta_{1}^{\#}\left(f_{2}\right)-\beta_{2}^{\#}\left(f_{2}+g\right)+\beta_{2}^{\#}\left(f_{2}\right)\right]\right\Vert _{L^{1}\left([0,T]\times\mathbb{R}^{n}\times\mathbb{R}^{n}\right)},$$

is equal or lower than 

$$\int_{0}^{T}\left[Q^{\#}\left(f_{2}+g,f_{2}+g\right)-Q^{\#}\left(f_{2}.f_{2}\right)\right]dt\left\Vert \beta_{1}^{\#}-\beta_{2}^{\#}\right\Vert _{L^{1}\left([0,T]\times\mathbb{R}^{n}\times\mathbb{R}^{n}\right)}.$$

Then

$\displaystyle\int_{0}^{T}\left[Q^{\#}\left(f_{2}+g,f_{2}+g\right)-Q^{\#}\left(f_{2}.f_{2}\right)\right]$

\begin{eqnarray*}
& = & \int_{0}^{T}\left[\int_{\mathbb{R}^{n}}\int_{\left|w\right|=1}B\left(w,\left|u-v\right|\right)\left.(f_{2}^{\#}\left(v'\right)f_{2}^{\#}\left(u'\right)+f_{2}^{\#}\left(v'\right)g^{\#}\left(u'\right)\right.\right.\\
		&&+\left.\left.g^{\#}\left(v'\right)f_{2}\left(u'\right)+g^{\#}\left(v'\right)g^{\#}\left(u'\right)+f_{2}^{\#}\left(v\right)f_{2}^{\#}\left(u\right)\right.\right.\\
		&&+\left.f_{2}^{\#}\left(v\right)g^{\#}\left(u\right)-g^{\#}\left(v\right)f_{2}^{\#}\left(u\right)-g^{\#}\left(v\right)g^{\#}\left(u\right)-f_{2}^{\#}\left(v'\right)f_{2}^{\#}\left(u'\right)+f_{2}^{\#}\left(v\right)f_{2}^{\#}\left(u\right)\right]dudwdt\\
&=&\int_{0}^{T}\int_{\mathbb{R}^{n}}\int_{\left|w\right|=1}B\left(w,\left|u-v\right|\right)\left[f_{2}^{\#}\left(v'\right)g^{\#}\left(u'\right)+g^{\#}\left(v'\right)f_{2}\left(u'\right)-f_{2}^{\#}\left(v\right)g^{\#}\left(u\right)-g^{\#}\left(v\right)f_{2}^{\#}\left(u\right)\right]dudwdt\\
		&&+\int_{0}^{T}\int_{\mathbb{R}^{n}}\int_{\left|w\right|=1}B\left(w,\left|u-v\right|\right)\left[g^{\#}\left(v'\right)g^{\#}\left(u'\right)-g^{\#}\left(v\right)g^{\#}\left(u\right)\right]dudwdt\\
&=&2\int_{0}^{T}Q^{\#}\left(f_{2},g\right)dt+\int_{0}^{T}Q^{\#}\left(g,g\right)dt\\
&=&\int_{0}^{T}Q^{\#}\left(g,g\right)dt\\
&<&1.
\end{eqnarray*}
	
So by Banach's Fixed Point Theorem, exists only a single fixed point of  $F$, that is $\beta^{\#}\left(f_{2}+g\right)-\beta^{\#}\left(f_{2}\right)=0$,
	we only are interested in the fixed points of  $\beta^{\#}$, therefore
	$f_{2}+g-f_{2}=0$ implies $g=0$, then $f_{1}=f_{2}$.
    \end{proof}

\end{document}